\documentclass[12pt]{amsart}
\usepackage{bm,amsmath,amssymb,eucal}
\usepackage{hyperref}
\input{comment.sty}

\newtheorem{Th}{Theorem}[section]
\newtheorem{Prop}[Th]{Proposition}%[section]
\newtheorem{Lem}[Th]{Lemma}%[section]
\newtheorem{Cor}[Th]{Corollary}%[section]

\theoremstyle{definition}
\theoremstyle{remark}
\newtheorem{Rem}[Th]{Remark}

\def\BN{\mathbb N}

\def\BR{\mathbb R}

\def\SA{\mathcal A}
\def\SB{\mathcal B}

\def\SE{\mathcal E}
\def\SF{\mathcal F}

\def\SI{\mathcal I}
\def\SJ{\mathcal J}
\def\SK{\mathcal K}

\def\SN{\mathcal N}

\def\si{\sigma}

\def\sminus{\smallsetminus}

\def\({\left(}
\def\){\right)}

\def\oo{\infty}

\begin{document}

\title{On Two Theorems of Sierpi\'nski}

\author[Grzegorek and Labuda]{Edward Grzegorek and Iwo Labuda}

\address{E. Grzegorek\endgraf
Institute of Mathematics, University of Gda\'nsk\endgraf
Wita Stwosza 57, 80--308 Gda\'nsk,
Poland}
\email{egrzeg@mat.ug.edu.pl}
\address{I. Labuda\endgraf
Department of Mathematics, The University of Mississippi\endgraf
University, MS 38677, USA}

\email{mmlabuda@olemiss.edu}

\date{}



\begin{abstract}  A theorem of Sierpi\'nski says that every infinite set  $Q$ of reals contains an infinite number of  disjoint subsets  whose outer Lebesgue measure is the same as that of $Q$. He also has a  similar theorem involving  Baire property. We give a general theorem of this type and its  corollaries, strengthening  classical results.    \end{abstract}

\maketitle

\section{Motivation and Definitions}

The main result of this note is motivated by two theorems of Sierpi\'nski which, in turn, are connected with the results and techniques due to Lusin and Novikov \cite{Lus34}.
The first theorem (\cite[Theorem II]{Sie34a}) is

\begin{Th}\label{th:Sierp2} Let  $Q$ be an infinite subset of $\BR$. Then $Q$ contains an infinite number of disjoint subsets, each of which has the same outer Lebesgue measure as $Q$.
\end{Th}
 The second  (see  Th\'eor\`eme 1 and a remark following it in  \cite[Suppl\'ement]{Sie34}) is

\begin{Th}\label{th:Sierp1} Let  $Q$ be a subset of an interval $J$.  If $Q$ is everywhere 2nd category in $J$, then it is the union of an infinite number of disjoint subsets, each of which is everywhere 2nd category in $J$.
\end{Th}

We recall that $Q$ is said to be everywhere 2nd category in $J$ if $Q$ is 2nd category at each point $x\in J$. The latter means that for each neighborhood $V$ of $x$, the intersection $V\cap Q$ is 2nd category.

When searching for a common abstraction of those results,  one needs to find a proper notion of the `size' of a subset $E$ in $Q$, a notion that would cover the measure case as well as  the Baire category case.

Our basic framework is a \textit{triple $(X,\SA,\SI)$}, where $X$ is a set, $\SA$ a $\si$-algebra of  its subsets (which are called \textit{measurable}), and $\SI$  a $\si$-ideal of its subsets (which are called \textit{negligible}) such that $\SI\subset \SA$.
 We will use the following definition that goes back at   least to \cite{Grz79} (in that paper the subset $E$ was qualified as \textit{being of the same size as} $Q$).

Let $E$ and $Q$ be subsets of $X$, with $E\subset Q$. We say that $E$ \textit{is a full subset of} $Q$ if, for each  $B\in\SA$ containing $E$, the set $B$ must  contain $Q$ modulo $\SI$, that is, one has $Q\sminus B\in\SI$.
In a situation in which one would like to specify the $\si$-ideal, we may also write that $E$ is an $\SI$-full subset of $Q$. One easily shows

\begin{Prop}\label{prop:full}  The following are equivalent.
\begin{enumerate}
\item[\rm{(a)}] $E$ is a full subset of $Q$.
\item[\rm{(b)}] For each $A\in\SA$, if $A\cap Q\not\in\SI$, then $A\cap E\not\in\SI$.
\end{enumerate}
\end{Prop}

%\begin{proof} (a) implies (b).
%Suppose, to the contrary, that there exists $A\in\SA$ such that $A\cap Q\not\in\SI$ %and $A\cap E\in\SI$. Then $C=A\sminus (A\cap E)$ is a  measurable set,
%is disjoint with $E$. and $C\cap A\not\in\SI$. As $C'$, the complement of %$C$, is a measurable set containing $E$, by (a), $Q\sminus C'=Q\cap C$ is %negligible; a contradiction.
%(b) implies (a).
%Suppose that (a) does not hold.
%Then there exists $B\in\SA$ containing $E$ and such that $Q\cap B'\not\in\SI$, %where $B'$ is the complement of $B$. Hence, by (c), $E\cap B'\not\in\SI$, %contradicting the fact that $B'\cap E=\emptyset.
%\end{proof}

 Following \cite{Zeb07}, let us say that a subset $E$ of $Q$ is  
$(\SA,\SI)$-\textit{completely non-measurable in $Q$} if the following condition is satisfied: for each measurable set $A$, if $A\cap Q \not\in\SI$, then $A\cap E\not\in\SI$ and $A\cap (Q\sminus E)\not\in\SI$.

In view of the condition (b) in the Proposition above, we have
\begin{Prop}\label{prop:doubly-full}Let $E\subset Q$.
Then  $E$ is $(\SA,\SI)$-completely non-measurable in $Q$ if\-f $E$ and $Q\sminus E$ are full subsets of $Q$.
\end{Prop}

 Consequently, if $\{E_1,E_2\}$ is a decomposition of a set  into disjoint full subsets, then it is its decomposition into $(\SA,\SI)$-completely non-measurable subsets. 

In what follows we  simply write \textit{completely non-measurable}, since  there will not be any ambiguity about $\SA$ and~$\SI$.

\section{Result}

Let $F$ be a subset of $X$.
We say that a set $\tilde F$ containing $F$ is a \textit{measurable envelope} of $F$ if $\tilde F\in\SA$ and, for each measurable $A$, the containment $A \subset\tilde F\sminus F$ implies $A\in\SI$. Below `envelope' will mean `measurable envelope'. It is easily seen that a measurable set $Q$ is an envelope of $E$ if\-f $E$ is full in $Q$.
Further, we have

\begin{Lem}\label{lem:hull} Let $Q$ be a measurable  set contained in  $\tilde F$. Then $Q$ is an envelope of $Q\cap F$. In particular, $Q\cap F$ is full in $Q$.
\end{Lem}

Indeed, let $A$ be a measurable set contained in $Q\sminus (Q\cap F)$. Then $A\subset\tilde F\sminus F$ and so $A\in\SI$. Hence $Q$ is an  envelope of $Q\cap F$.

\medskip

Let a triple $(X,\SA,\SI)$, as defined in  Section 1, be given. \textit{ Throughout this Section we     impose the following two assumptions on} $(X,\SA,\SI)$.

 \begin{itemize}
\item The family $\SA\sminus\SI$ satisfies the countable chain condition CCC, that is, every disjoint family of non-negligible measurable sets  is countable. It is well-known and  easily shown that under CCC each subset of $X$ admits an envelope.
\item $X$ has the \textit{non-separation property with respect to $\SA$}. Namely, if $E\not\in\SI$, then there exist disjoint subsets  of $E$ which cannot be separated by  sets from $\SA$ (in particular, the singletons must belong to $\SI$).
\end{itemize}

\begin{Lem}\label{lem:main}
If $E\not\in\SI$, then there exist
a measurable set $Q\not\in\SI$  and disjoint sets $E_1, E_2$ contained in $E\cap Q$ such that $E_1$ and $E_2$ are full subsets of $Q$.  \end{Lem}

\begin{proof} Let $F$ and $G$ be disjoint subsets of $E$ that cannot be separated by measurable sets. Denote by  $\tilde F$  and $\tilde G$   their respective envelopes. Put $Q=\tilde F\cap\tilde G$.  As $F$ and $G$  are not separated,  $\tilde F\cap G\not\in\SI$, since otherwise $\tilde F\sminus (\tilde F\cap G)$  would be a measurable set separating $F$ and $G$.  Hence $Q=\tilde F\cap\tilde G$ is not negligible. By Lemma~\ref{lem:hull}, $Q$ is an envelope of  $F\cap Q$ and also of $G\cap Q$. Thus, $E_1= F\cap\ Q$ and $E_2=G\cap Q$ are full in $Q$.
\end{proof}

\begin{Rem} Envelopes are also instrumental in characterizing fullness.  Notably, the following Proposition holds:\textit{ Let $E\subset Q\subset X$ and assume that $Q$ admits an envelope. Then  $E$ is a full subset of $Q$ if\-f any envelope of $Q$ is an envelope of $E$.}
\end{Rem}

Here is our main result.

\begin{Th}\label{th:main} Each infinite subset $E$ of $X$ decomposes into  two disjoint full subsets. Consequently, it can be written as the infinite disjoint union of its full (equivalently -- completely non-measurable) subsets.
\end{Th}

\begin{proof}
We first show the decomposition into two subsets. We may assume that $E\not\in\SI$.

Let us say  that a set \textit{ $P$  has the  equal sizes property relative to  $E$} whenever $P$ is a non-negligible measurable  set    such that there exist disjoint sets $E_1, E_2$ contained in $E\cap P$, each one full in $ P$.

Let  $\SF$ be a maximal disjoint family of sets $\{Q_\xi\}$ having that property. It is not empty by Lemma~\ref{lem:main}.  By   the countable chain condition, $\SF=Q_1, Q_2,\dots$. Observe first that $$
E^*= E\sminus\bigcup_{n=1}^\infty Q_n\in\SI. \leqno(*)
$$
Indeed, otherwise one could apply Lemma~\ref{lem:main} to $E^*$ and find a corresponding  set $Q^*$.  That is,
$Q^*$ belongs to $\SA\sminus\SI$ and
there exist disjoint subsets $E_1$ and $E_2$ of $E^*\cap Q^*$ which are both full in $Q^*$. We thus have $E_1$ and $E_2$ contained in $A$, where $A= Q^*\sminus \bigcup_n Q_n$.
 As $A\subset Q^*$, the sets $E_1, E_2$ are full in $A$. Also,   $A\not\in\SI$ since it contains $E_1$ which is not in $\SI$. To sum up, $A$ is disjoint with all $Q_n$'s and has the equal sizes property relative to $E$. This contradicts the fact that   $\SF$ is maximal.

Now, for each $n\in\BN$,  let $C_n$ and $D_n$ be  disjoint  subsets of $E\cap Q_n$ that, moreover, are full subsets of $Q_n$. We claim that  $C=\bigcup_n C_n$ and $D=\bigcup D_n$ are full subsets of $E$. Indeed,  let $C$ be contained in a measurable set $B$. Then $C_n\subset B$ and so $Q_n\sminus B\in\SI$ for each $n\in\BN$. This implies that $(\bigcup Q_n)\sminus B\in\SI$ and finally, by $(*)$, that $E\sminus B\in\SI$. For  $D\subset B$, the  argument is similar.
If $D$ is not the complement of $C$ in $E$, we can replace it by that complement in order to get the needed decomposition.

By what we have just shown, $E$ is the union of two full subsets, $E_1$ and its complement $E_1'$ in $E$. Next, by the same argument, decompose $E_1$ into $E_2$ and $ E_2'$. It can readily be checked that $E_2$, as a full subset of a full subset of $E$, is full in $E$. Proceed inductively.
\end{proof}

\begin{Rem}\label{rem:rem}   The proof above is a modification of Sierpi\'nski's argument in the proof of Th\'eor\`eme I in \cite{Sie34a}. It is Lemma~\ref{lem:main} that makes the passage to the abstract  case possible. For  the Lebesgue measure, Lusin stated in \cite{Lus34} a fact stronger than Lemma~\ref{lem:main} (see \cite[Appendix]{Grze17}) and it was this stronger fact that Sierpi\'nski used in his proof of Theorem~\ref{th:Sierp2}.
\end{Rem}

Consider  triples $(X,\SA_n,\SJ_n)$, $\,n\in\BN$, where $\SA_n$ are  $\si$-algebras and $\,\SJ_n\subset\SA_n$ are $\si$-ideals  in $X$, together with $(X,\SA^\oo,\SN),$  where the $\si$-algebra $\SA^\oo=\bigcap_n\SA_n$ and the $\si$-ideal $\SN=\bigcap_n\SJ_n$. Assume that the triples $(X,\SA_n,\SJ_n)$ satisfy the  assumptions specified by the bullets  above.

\begin{Prop}\label{prop:stability}  The family  $\SA^\oo\sminus\SN$ satisfies CCC and $X$ has the non-separation property with respect to~$\SA^\oo$.
\end{Prop}
\begin{proof}
Let $\SE$ be a family of disjoint sets in $\SA^\oo$ that do not belong to $\SN$. Then, $\SE=\bigcup_{n=1}^\oo\SE_n$, where $\SE_n\subset\{A\in\SA^\oo: A\not\in\SJ_n\}\subset \{A\in\SA_n: A\not\in\SJ_n\}$. Hence each $\SE_n$  is countable, and so $\SE$ is countable as well.

Suppose  $E$ is a subset of $X$ such that  $E\not\in\SN$.  Then,  $E\not\in\SJ_n$ for some $n\in\BN$. Consequently, there exist two disjoint subsets of $E$ that cannot be separated by sets from $\SA_n$. A fortiori, they cannot be separated by sets from $\SA^\oo$.
\end{proof}

\begin{Cor}\label{cor:intersection}
Each infinite subset $E$ of $X$ decomposes into  two disjoint $\SN$-full subsets. Consequently, it can be written as the infinite disjoint union of its subsets, each of which is $\SN$-full (equivalently -- completely non-measurable with respect to to $\SA^\oo$).
\end{Cor}

Let $E=\bigcup_{i=1}^\oo E_i$ be the decomposition  obtained for the set $E$ in the above Corollary.

\begin{Prop}\label{prop:sim} Assume that, for each $n\in\BN$ and each $A\in \SA_n$, there exists $B\in\SA^\oo$ such that $A\subset B$ and $B\sminus A\in\SJ_n$.
Then each $E_i$ in the decomposition  is $\SJ_n$-full (equivalently -- completely non-measurable with respect to $\SA_n$) simultaneously for all $n\in\BN$.
\end{Prop}

\begin{proof}
Let $C$ be a subset of $Q$ that is full with respect to $\SN$ and  $\SA^\oo$. Fix $n$ and take $A\in\SA_n$ containing $C$. Find $B\in\SA^\oo$ containing  $A$ and such that $B\sminus A\in\SJ_n$. As $B\supset C$, one has
  $Q\sminus B\in\SN\subset \SJ_n$.  Consequently, $Q\sminus A\in\SJ_n$ and $C$ is $\SJ_n$-full in $Q$.
\end{proof}

\begin{Rem}

Let $\SB$ be a $\si$-algebra,  and $\SJ$ a $\si$-ideal in $X$. Then $\SJ$ is said to have \textit{a base in} $\SB$ if, for each $J\in\SJ$, there exists $K\in\SJ$ containing $J$ and belonging to  $\SB$.

Suppose that $\SA_n=\SB\vartriangle \SJ_n$, where $\vartriangle$ stands for the symmetric difference. It is easy to see that, if $\SJ_n$ has a base in $\SB$, then, for each $A\in\SA_n$, there exists $B\in\SB$ containing $A$ and such that $B\sminus A\in\SJ_n$. In particular, the assumptions in the above Proposition are satisfied.
\end{Rem}

\section{Applications}

 \ \ {\bf \S1. $\mu$-measurability.} Let $\mu$ be a complete $\si$-finite (positive, countably additive) measure on a set $X$. Let $\SA$ be the $\si$-algebra  of $\mu$-measurable sets and  $\SI=\SJ$ its $\si$-ideal of $\mu$-zero sets. Write $\mu^*$ the outer measure of $\mu$, and let $E$ and $Q$ be subsets of $X$ with $E\subset Q$.  We note that if $\mu^*(Q)<\oo$, then $E$ is $\SJ$-full subset of $Q$ if\-f $\mu^*(E)=\mu^*(Q)$.

\begin{Cor}\label{cor:measure} Let $X$ be a 2nd countable topological $T_1$-space, and $\mu$ the completion of a regular Borel $\si$-finite  measure  vanishing on points. Then every infinite subset $E$ of $X$ can be decomposed into an infinite disjoint union of  its $\SJ$-full (equivalently -- completely non-measurable) subsets. In particular, each of these subsets has the same outer measure as~$E$.
\end{Cor}
\begin{proof}
It is well known that  $\SA\sminus\SJ$ has CCC. The non-separation property of $X$ with respect to $\SA$ follows from the Lusin-Novikov Theorem \cite[Corollary 5.8]{Grze17}.
\end{proof}

{\bf\S2. Baire property.} Let us recall that a subset $E$ of a topological space $X$ \textit{has the Baire property} (or \textit{is $BP$-measurable})    if it can be written in the form
$
E=(O\sminus P)\cup Q,
$
where the set $O$ is open, while $P$ and $Q$ are 1st category.  Let $\SA=\SB\vartriangle\SI$ be the $\si$-algebra of sets having the Baire property in $X$, where $\SB$ denotes  the $\si$-algebra of  Borel sets and  $\SI=\SK$ is the $\si$-ideal of 1st category (meager) sets.  If $X$ is 2nd countable, then $\SA\sminus\SK$ has CCC (\cite[\S24.I,\,Theorem 3]{Kur66}).   If  $X$ is a 2nd countable $T_1$-space without isolated points, then $X$ has the non-separation property with respect to $\SA$
 by  the Lusin-Novikov Theorem (see \cite [Corollary 3.2]{Grze17}). Let us add that the existence of envelopes does not depend on CCC. Every subset $F\subset X$ has  an $F_\si$-envelope $\tilde F$ (Szpilrajn-Marczewski's theorem \cite[\S11.IV, Corollary 1]{Kur66}).

In the `Baire property case' completely non-measurable subsets of a set $Q$  were called \textit{completely non-Baire subsets of $Q$} in \cite{Cich07}. Following  this terminology, we have
\begin{Cor}\label{cor:baire} Let $X$ be a 2nd countable topological $T_1$-space without isolated points. Then every infinite subset of $X$ can be decomposed into an infinite disjoint union of its $\SK$-full (equivalently -- completely non-Baire) subsets.
\end{Cor}

\begin{Rem}
  If $X$ is an uncountable Polish space, a generalization of the above Corollary for partitions into 1st category sets holds (\cite[Theorem 6.8]{Cich07}).
\end{Rem}

We recall that a topological space is called a \textit{Baire space} if its nonempty open subsets are 2nd category.
\begin{Prop} If $E\subset X$ is 2nd category at each point $x\in X$, then $E$ is full in $X$. Conversely, if $X$ is a Baire space and $E$ is full  in $X$, then $E$ is 2nd category at each point $x\in X$. 
\end{Prop}

\begin{proof} For a subset $A$ of $X$, denote by $D(A)$ the set of all points in which $A$ is 2nd category. Let $B$ be a set having the Baire property and containing $E$. Then, by \cite[\S 11.\,IV.5]{Kur66}, $D(B)\sminus B$ is of 1st category. Since $D(B)\supset D(E)=X$, we have  that $X\sminus B$ is 1st category.  

To see the converse, observe that, by the condition (b) of Proposition~\ref{prop:full}, 
$E$ is full in  a set $Q\subset X$ means that for each $A\in\SA$, if $A\cap Q\not\in\SI$, then $A\cap E\not\in\SI$. It is not difficult to show, that it is sufficient  to take open sets in the above condition. Thus, if $Q=X$, we have that $E$ is full in $X$ means that for each open set $A$ which is of 2nd category, $E\cap A$ is 2nd category. Let $V$ be an open neighborhood of $x\in X$. As $X$ is assumed to be a Baire space, $V$ is 2nd category, and so $V\cap E$ is 2nd category as well. 
\end{proof}

Theorem~\ref{th:Sierp1}, a solution to a problem posed by Kuratowski \cite[Probl\`eme 21]{Kur23}, can be generalized as follows.

\begin{Cor} 
Let $X$ be a 2nd countable topological Baire $T_1$-space without isolated points and $E\subset X$.   If $E$ is everywhere 2nd category in $X$, then it can be decomposed into the union of an infinite number of disjoint subsets, each of which is everywhere 2nd category in $X$.
\end{Cor}

\begin{proof} If $E$ is everywhere 2nd category in $X$, then it is full there. Applying Corollary~\ref{cor:baire}, $E$ can be decomposed into its two full subsets $E_1$ and $E_2$ , which are therefore full in X. Hence they are both everywhere 2nd category in $X$. Proceed inductively.
\end{proof}

{\bf \S3. Other examples.}

In view of Proposition~\ref{prop:sim} and the remark following it, we   have

\begin{Cor} With the assumptions of Corollaries~\ref{cor:measure} and \ref{cor:baire} combined, each infinite subset of $X$ admits a countable decomposition  into  subsets that are full (or completely non-measurable)  in terms of measure and Baire property simultaneously.
\end{Cor}

Under additional set-theoretical assumptions about $X$,  decompositions into more than a countable number of subsets are possible. Actually, the main result  of \cite{Sie34a}, the already mentioned Th\'eor\`eme I, is concerned with  an uncountable case.  We now give a  general result of this type.

Let $A$ be a $\si$-algebra of subsets of a set $X$ containing all singletons. Define $\SI=\SI(\SA)$ as the $\si$-ideal of those sets in $\SA$  which are hereditarily in $\SA$, i.e., if $A$ belongs to $\SA$ together with each subset  $C$ of $A$, then $A\in\SI$. Taking into account \cite[Theorem 2]{Grz79}, we have

\begin{Cor}  Suppose the cardinality of $X$ is less than the first weakly inaccessible cardinal and $\SA\sminus \SI$ satisfies CCC. Then each uncountable subset  $Y$ of $X$ can be decomposed into $\aleph_1$ disjoint full (or completely non-measurable) subsets.
\end{Cor}

Further, we were informed by the referee of this paper that, in a separable metric space, under the non-existence of quasi-measurable cardinals less or equal $2^{\aleph_0}$, for a CCC $\si$-ideal $\SI$ with a Borel base every point-finite family $\SA\subset\SI$  can be decomposed into $\aleph_1$ subfamilies, the union of each of them being full in $\bigcup\SA$ relative to $\SI$  \cite[Theorem 3.6]{RZ}. Also, under some cardinal coefficients, any $cf(2^{\aleph_0})$-point family $\SA$ of subsets of a $\si$-ideal $\SI$ can be decomposed into $2^{\aleph_0}$ pairwise disjoint subfamilies, the union of each of them being full in $\bigcup\SA$ relative to $\SI$ (\cite[Corollary 2.1]{RZ}.

\section{Appendix}

\subsection{Historical comments}

 (1) The title of the present note refers to \textit{two theorems of Sierpi\'nski} for the simple reason that these theorems were first published by Sierpi\'nski. However, it ought to be mentioned that  in \cite[Suppl\'ement]{Sie34} Sierpi\'nski considers his  Th\'eor\`eme 1 as an easy consequence of a `method of Lusin' and in \cite[\S10\,(p.\,87)]{Kur66} the result reappears as `Lusin Theorem' with a sole reference to the Suppl\'ement. Both attributions strike us as being not correct. Admittedly, it seems possible that writing his comments in the Suppl\'ement Sierpi\'nski could have relayed   on a letter from Lusin  quoted there and did not yet know \cite{Lus34}. However, in  \cite{Sie34a} Sierpi\'nski refers directly to \cite{Lus34}, in which paper Lusin himself admits that the method of proof is due to Novikov; and still only the name of Lusin is mentioned in\cite{Sie34a}. Similarly, writing much later, Kuratowski must have known \cite{Lus34} and its reviews \cite{Chit34}, \cite{Ros34}, in which the proofs of \cite{Lus34} are referred to as  Novikov's. Yet, he keeps quoting only the Suppl\'ement and calls the result `Lusin Theorem'.
 What would be the reason for this silence about Novikov, we do not know.

 (2)
Ashutosh Kumar stated  in his thesis \cite[Theorem 1]{Kum14}  the decomposition of a  subset of $\BR^n$ into two disjoint full subsets in the Baire property case and the Lebesgue measure case. The Baire part of that statement is a new result corresponding to our Corollary~\ref{cor:measure}.
Kumar attributes the whole theorem to Lusin quoting \cite{Lus34}, despite the fact that no such statement is present in \cite{Lus34} and despite the fact that any result quoted from \cite{Lus34} should be attributed to Lusin and Novikov.   He repeats the misquotation in \cite{Kum13} and then, together with Shelah, in \cite{Kum17}. We are also aware of the work by Marcin Michalski (although no preprint is available yet), who in his talk \cite{Mich17} gives a general definition of an $\SI$-full subset and announces a proof of the Lebesgue measure and the Baire property statements in $\BR$ by using a modification of the technique of \cite{Lus34}, still attributing the result to Lusin.

\subsection{Kumar's proof}

Consider a triple $(X,\SA,\SI)$ introduced in Section 1 and assume that $\SA$-measurable envelopes exist in the space $X$. Then, the following proposition holds.

\begin{Prop}\label{prop:equiv}

For  a triple $(X,\SA,\SI)$ and a subset $Y$ of $X$,  the following conditions are equivalent.
\begin{enumerate}
\item[\rm{(a)}]  $Y$ contains  $Z\not\in \SI$  such that $Y\sminus Z$ is full in $Y$.
\item[\rm{(b)}]  $Y$ contains  $Z\not\in\SI$  such that $Z$ and $Y\sminus Z$ cannot be separated by an $\SA$-measurable set.

\item[\rm (c)]  $Y$ contains disjoint subsets $F$ and $G$ that cannot be separated by  an $\SA$-measurable set.
\end{enumerate}
\end{Prop}

\begin{proof} The implications from (a) to (b) and from (b) to (c) being evident, only (c)  implies (a) has to be shown.

Assume (c), and define $Q=\tilde F\cap \tilde G$. By applying Lemma~\ref{lem:hull} (with $Y= E$), we have $Q\cap F$ and $Q\cap G$, subsets of $Y\cap Q$, which are full in $Q$. Then (a) holds with $Z=Q\cap F$. 

Indeed, we note that $Y\cap (Q\sminus F)$, containing $Q\cap G$ that is full in $Q$, must be full in $Q\cap Y$. Obviously $Y\sminus Q$ is full in itself, and so  $(Y\sminus Q)\cup [Y\cap(Q\sminus F)]$ is full in $(Y\sminus Q)\cup(Y\cap Q)$. Hence $Y\sminus (Q\cap  F)$ is full in $Y$.  
\end{proof}

 Kumar's proof of his Theorem 1 in \cite{Kum14} proceeds first with the  measure case. He begins by remarking that one can reduce to consider the Lebesgue measure on $[0,1]$ and says that it is enough to show that every non-negligible subset $Y\subset [0,1]$ satisfies the condition (a) above. This is so, because then, one can find the needed decomposition of $Y$ into two full subsets. He   defines the needed full subset $E$ without actually showing that the so defined $E$ and its complement are indeed full in $Y$. In the next step the proof  becomes indirect. By denying the needed condition, he concludes that every two disjoint subsets of $Y$ would be separated by a measurable set (in his phrasing the restriction of the outer  Lebesgue measure to $Y$ is countably additive on the power set of $Y$).  At this point the proof stops being classical. By a  use of forcing, it is shown that the above conclusion about the outer Lebesgue measure on $Y$  leads to a contradiction. 

The proof of the Baire case is similar (although the definition of $E$ is not  stated and so the proof that it has required properties is not given either). But assuming that $E$ can be well defined, which Kumar silently does,  by analogy with the measure case one can find a 2nd category subset $Y$ of $[0,1]$ such that for every 2nd category subset $E$ of $Y$ its complement $Y\sminus E$ is not full in $Y$. Then, an appropriately modified use of forcing gives a contradiction. 

Here are our comments.

(1) The measure case. It is the contradiction proved by forcing that is actually proven classically in the the quoted note of Lusin \cite{Lus34}.

(2) Kumar's  reasoning   showing the existence of a subset $Y$ whose every two disjoint subsets can be separated by a Lebesgue measurable set remains the same in the general framework of $\SA$-measurability. It shows that, if for each non-negligible subset $Z$ of $Y$ one has $Y\sminus Z$ which is not full in $Y$, then every two disjoint subsets of $Y$ can be separated by an $\SA$-measurable set. Thus, also in the Baire case the contradiction gotten by forcing is  a consequence of a result in \cite{Lus34}

(3) The (omitted by Kumar) argument showing that $E$ and its complement are full in $Y$ can be adapted to the general $\SA$-measurable case. But  there is no point giving details, since it follows from Proposition~\ref{prop:equiv} that our approach and Kumar's are equivalent.

.


\begin{thebibliography}{11}

\bibitem{Chit34}
E. W. Chittenden,
\textit{Review of \cite{Lus34}}, Zentralblatt f\"ur Mathematik (zbmath.org: document no. 0009.10402).

\bibitem{Cich07}
J. Cicho\'n, M. Morayne, R. Ra\l owski, C. Ryll--Nardzewski and S. \.Zeberski,
\textit{On nonmeasurable unions}, Topology and Appl., {\bf 154} (2007), 884--893.
\bibitem{Grz79}
E. Grzegorek,
\textit{On a paper by Karel Prikry concerning Ulam's problem on families of measures},
Coll. Math., {\bf 52} (1979), 197--208.



\bibitem{Grze17}
E. Grzegorek and I. Labuda,
\textit{Partitions into thin sets and forgotten theorems of Kunugi and Lusin-Novikov},
 to appear in Coll. Math.


\bibitem{Kum14}
A. Kumar,
\textit{On some problems in set-theoretical analysis}, PhD Thesis, University of Wisconsin-Madison 2014, 35 pp. http://www.math.huji.ac.il/~akumar/thesis.pdf

\bibitem{Kum13}
A. Kumar,
\textit{Avoiding rational distances}.
Real Anal. Exchange, {\bf 38(1)}, (2012/13), 493--498.

\bibitem{Kum17}
A. Kumar and S. Shelah,
\textit{A transversal of full outer measure}, 

 http://www.math.huji.ac.il/akumar/tfom.pdf.

\bibitem{Kur23}
K. Kuratowski,
\textit{Probl\`emes},
Fund. Math. {\bf 4} (1923), 368--370.
\bibitem{Kur66}
K. Kuratowski,
\textit{Topology, Vol. 1},
Academic Press--PWN 1966.

\bibitem{Lus34}
N. Lusin,
\textit{Sur la d\'ecomposition des ensembles},
C. R. Acad. Sci. Paris, {\bf 198} (1934), 1671--1674.

\bibitem{Mich17}
M. Michalski,
\textit{Odkopane twierdzenie \L uzina}, III Warsztaty z Analizy Rzeczywistej, 20-21 May 2017, Konopnica; http://www.im.p.lodz.pl/semwa/pliki/MMichalski2017.pdf.

\bibitem{RZ}
R. Ra\l owski and S. \.Zeberski,
\textit{Completely measurable families},
Central European J. Math.. {\bf 8}(4), (2010), 683--687. 




\bibitem{Ros34}
A. Rosenthal,
\textit{Review of \cite{Lus34}}, Jahrbuch \"uber die Fortschritte der Mathematik (zbmath.org: document no. 60.0039.01).

\bibitem{Sie34}
W. Sierpi\'nski,
\textit{Hypoth\`ese du continu},
Monografje Matematyczne, Warszawa--Lw\'ow 1934.

\bibitem {Sie34a}
W. Sierpi\'nski,
\textit{Sur une propri\'et\'e des ensembles lin\'eaires quelconques},
Fund. Math., {\bf 23} (1934), 125--134.

\bibitem{Zeb07}
S. \.Zeberski,
\textit{On completely nonmeasurble unions},
Math. Log. Quart. {\bf53(1)} (2007), 38--42.
\end{thebibliography}
\end{document}